\def\Bbb{\mathbb}

\def\dist{{\rm{dist}}}

\def\dist{{\rm{dist}}}

\def\logcap{{\rm{cap}}}
\def\Bbb{\mathbb}

\def\complex{\Bbb C}
\def\disk{\Bbb D}
\def\circle{\Bbb T}

\def\disk{\Bbb D}
\def\cal{\mathcal}

\documentclass[12pt]{amsart}  
\usepackage{amssymb}    

\usepackage{graphicx,amssymb,amsthm,verbatim,amsfonts}

\theoremstyle{plain}                    
\newtheorem{thm}{Theorem}[section]
\newtheorem{cor}[thm]{Corollary}
\newtheorem{lemma}[thm]{Lemma}

\newcounter{ques}
\newtheorem{ques}[thm]{Question}   
\numberwithin{equation}{section}

\begin{document}
\baselineskip=18pt


%

\title  [True trees are dense] 
         {True trees are dense}

\author {Christopher J. Bishop}
\address{C.J. Bishop\\
         Math. Dept. \\
         Stony Brook University \\
         Stony Brook, NY 11794-3651}
\email {bishop@math.sunysb.edu}
\thanks{The  author is partially supported by NSF Grant DMS 13-05233.
        }

\date{August 2013}
\maketitle


\begin{abstract}
We show that any compact, connected set $K$ in the plane can 
be approximated by the critical points of a polynomial with 
two critical values. Equivalently, $K$ can be approximated
in the Hausdorff metric 
by a true tree in the sense of Grothendieck's {\it dessins d'enfants}.
\end{abstract}

\clearpage


\setcounter{page}{1}
\renewcommand{\thepage}{\arabic{page}}
\section{Introduction} \label{Intro} 

Polynomials with at most two critical values are
called generalized Chebyshev polynomials or Shabat polynomials.
If $p$ is such a polynomial  of degree $n$  
with  critical values in  $\{\pm 1 \} $,  
then it is not hard to see that 
$T = p^{-1}([-1,1])$ is a finite planar tree with $n$
edges. We call a tree of this form a ``true tree''
or the ``true form'' of the combinatorial planar tree $T$.
True trees can have all possible combinatorics, i.e., 
every finite planar tree has a true 
form and this true form is unique  up to orientation
preserving  Euclidean similarities
(see Section 2).
 Can true trees attain all possible  ``shapes''?  More
precisely, given a continuum (i.e., a compact, connected set) in 
the plane, can we find a true tree that approximates it as 
closely as we wish?
The Hausdorff distance between two sets  is the minimum $\epsilon>0$ 
so that  each set is contained in an $\epsilon$-neighborhood of the other.
In this note we prove:

\begin{thm} \label{poly}
For any compact, connected set $K \subset \complex$ and any $\epsilon >0$ there
is a polynomial $p(z)$ with critical values exactly $\pm 1$  so that 
$T = p^{-1}([-1,1])$ 
approximates $K$ to within $\epsilon$ in the Hausdorff metric.
In other words, true trees are dense in all  planar continua.
\end{thm} 

True trees are 
a special case of Grothendieck's theory of {\it dessins 
d'enfants} in which a finite graph drawn on a compact
topological surface  
$X$ induces a conformal structure on the surface and a  Belyi
map to the  Riemann sphere (i.e., a meromorphic map 
branched over three points). 
In the case of a tree drawn on the plane, the compact 
surface is the Riemann sphere and  the Belyi map is a polynomial 
with two finite critical values ($\infty$ is the third branch
point). These maps have close connections to algebraic number
theory and Galois theory, although  we will not deal with those 
topics here.   There is an extensive
 literature on dessins d'enfants, 
true trees and Belyi functions, e.g.,  see  
\cite{MR2053391},
\cite{MR2349672},
 \cite{MR1746434}, 
 \cite{MR2476033},
\cite{MR2476034},
\cite{MR2411966},
\cite{MR1625545}, 
\cite{MR1305390},
\cite{MR1310587},
\cite{MR2310190}
 and their references.

Our approach  to proving Theorem \ref{poly} 
 is based on  interpreting true trees in terms of 
conformal maps. We will describe this alternate formulation
and reduce the  theorem to a more geometric sounding 
statement.

Suppose $T$ is a finite tree in the plane  with $n$ edges. Then the 
complement  $\Omega$ of $T$ is the image 
 of a  conformal map $f$ from $\disk^* = \{ |z| >1\}$ to   $\Omega$
with $f(\infty) = \infty$. We say that $T$ is ``conformally 
balanced'' if every  open edge of the tree is the image 
under $f$ of two disjoint open arcs of length $\pi/n$ on $\partial \disk^*$,
and   
$f(z) = f(w) $ implies $f'(z) = f'(w)$ for almost every 
$z, w \in \circle$. 
Because conformal maps preserve harmonic measure, the 
conformally balanced condition can be restated in terms 
harmonic measure with respect to $\infty$ on $T$ (i.e., 
the first hitting distribution of Brownian motion on the 
sphere started at $\infty$ and run until it hits $T$).
On each edge of the tree, harmonic 
measure naturally decomposes as the sum of two measures, one
corresponding to each side of the edge. The tree is 
conformally  balanced if (1)  every edge has the same harmonic 
measure, and (2) when we decompose harmonic measure on each edge 
into measures corresponding to the two  sides, 
 these two measures 
are identical. Note that we mean that these measures
 give the same mass
to every measurable subset of the edge, not merely  that the 
whole edge gets equal harmonic measure from both sides.

To see that conformally balanced trees are exactly the same as 
the true trees described above, 
suppose $T$ is a conformally balanced tree and let
 $f$ be  a conformal map from  $\disk^* = \{|z| >1\}$
to  $\Omega = \complex \setminus T$,
 preserving $\infty$ and such that 
 $1 \in \circle =\{|z|=1\} = \partial \disk^*$
 maps to a vertex.
Let $g(z) = \frac 12(z + z^{-1})$. This is called 
the Joukowsky map and is   the 
  conformal map from $\disk^*$ to  
 $ U = \complex \setminus [-1,1]$  that  fixes $-1,1,\infty$.
Each edge of $T$ has two  preimages under $f$ 
 of length $ \pi/d$ on $\circle$. 
Under the map $z \to z^d$ each interval is mapped to either the upper 
or lower half-circle and pairs of intervals corresponding to the same 
edge of the tree map to opposite half-circles. Points that are 
mapped by $f$ to the same point are also identified by $g$. Thus the 
map $g((f^{-1}(z))^d) $ defines a $d$-to-$1$ holomorphic map from the 
complement of the tree to the complement of $[-1,1]$. 
This map extends continuously 
to the whole plane and hence is a $d$-to-$1$  entire 
function  (see Lemma \ref{removable})
 and hence is a polynomial.
The critical points of $p$ are the vertices of degree $>1$  of
the  tree, the only critical 
values are $-1$ and $1$ and the tree itself is $p^{-1}([-1,1])$. 
See Figure \ref{Pcovers3}. 
The argument can be reversed, so  we see that true trees are 
the same as conformally balanced trees.
Thus Theorem  \ref{poly}  can be rewritten as

\begin{thm} \label{Main}
For any compact, connected set $K$ and any $\epsilon >0$ there
is a conformally  balanced tree $T$ that is within $\epsilon$ of 
$K$ in the Hausdorff metric.
\end{thm} 

\begin{figure}[htb]
\centerline{
	\includegraphics[height=3in]{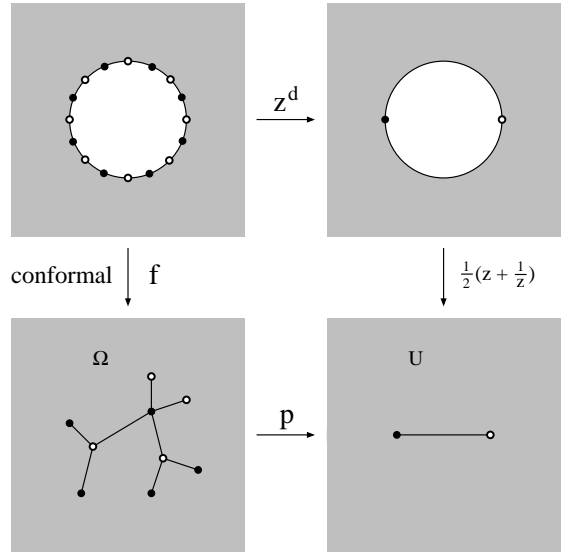}
}
\caption{ \label{Pcovers3}
For any tree $T$, the composition of  the conformal map 
from $\Omega = \complex \setminus T$  to $\{|z|>1\}$, 
followed by $z^d$, followed by $\frac 12  (z +\frac 1z)$ is 
$d$-to-$1$ and holomorphic off $T$. $T$ is conformally 
balanced (i.e., is a true tree) iff  this map extends 
continuously across $T$ and hence defines a polynomial
with critical values in $\{-1,1\}$. 
}
\end{figure}

Theorems  \ref{poly} and \ref{Main} were conjectured
by Alex Eremenko. I thank him for the enlightening discussion of 
these problems during his visit to Stony Brook in March 2011.
I thank Lasse Rempe for his comments on an earlier draft of this 
note.  I also thank the referee for a careful reading of the 
manuscript and numerous corrections and suggestions for 
improving the paper.

 Kevin Pilgrim has observed that the results in this paper, 
combined with his  arguments  
in \cite{MR1834499}, prove that Julia sets
of post-critically finite polynomials are dense in all 
planar continua.  The details will appear in \cite{Bishop-Pilgrim}. 
A related result was given using different 
methods by Kathryn Lindsey and William Thurston in 
\cite{Lindsey-Thurston}.

For the Shabat polynomials $p_n(z)= 2 z^n -1$, the only critical 
point is $0$, whereas the corresponding trees $T_n= p_n^{-1}([-1,1])$ become 
dense in the  unit disk as $n$ increases. Thus it is possible for  a 
sequence of true trees 
to approximate a set $K$, but the corresponding sets of 
 critical points not to approximate $K$. However,  
it will be clear from our construction that the set $K$ in 
the theorems  is 
approximated by  both a true tree and  its corresponding
set of critical points (i.e.,  the 
vertices of $T$ of degreee $> 1$).
Moreover, the trees we construct will  have a number 
of ``bounded geomery'' properties, e.g., the  maximum  vertex
degree  is $4$  and the edges are all analytic arcs 
with uniform  estimates (i.e., each edge is  the 
image of $I=[-1,1]$ under a map that is conformal on 
a uniform neighborhood $\Omega$ of $I$).

Don Marshall and Steffen Rohde  have recently  adapted
Marshall's conformal mapping program \verb+zipper+ to 
approximate the true form of a given planar tree, 
\cite{Marshall-Rohde}.
 The program can handle 
examples with thousands of edges and is highly accurate. 
See Figure \ref{Marshall} for some examples.

\begin{figure}
\centerline{
\includegraphics[height=2.25in]{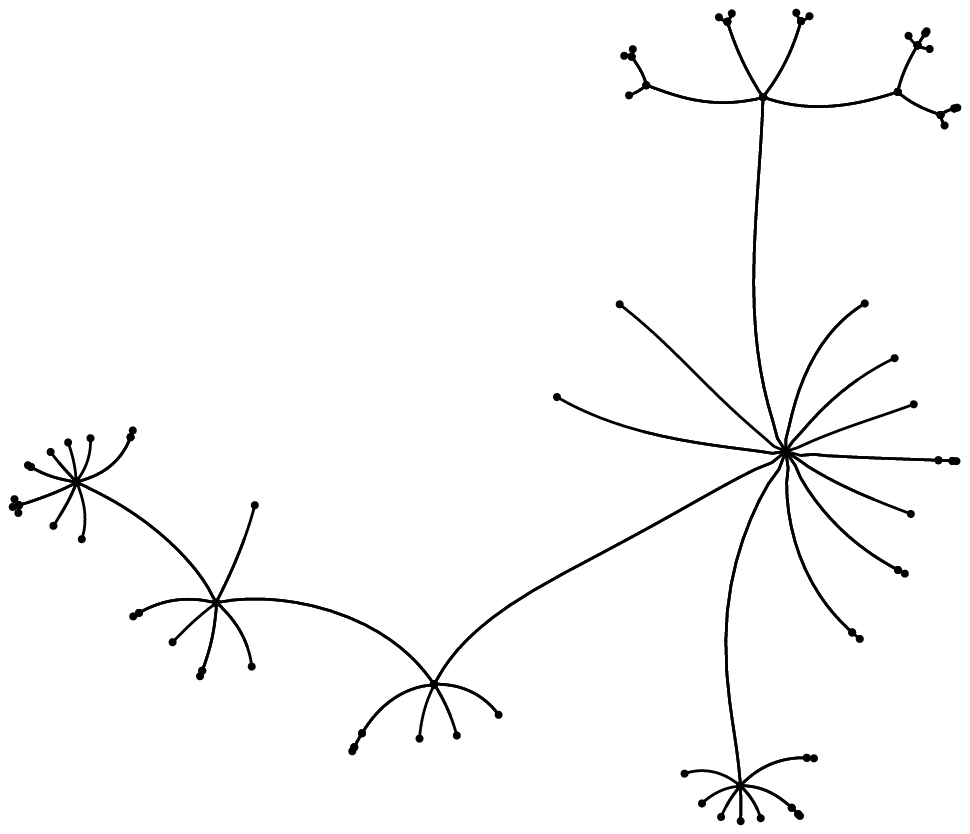}
$\hphantom{xxxxx}$ 
\includegraphics[height=2.25in]{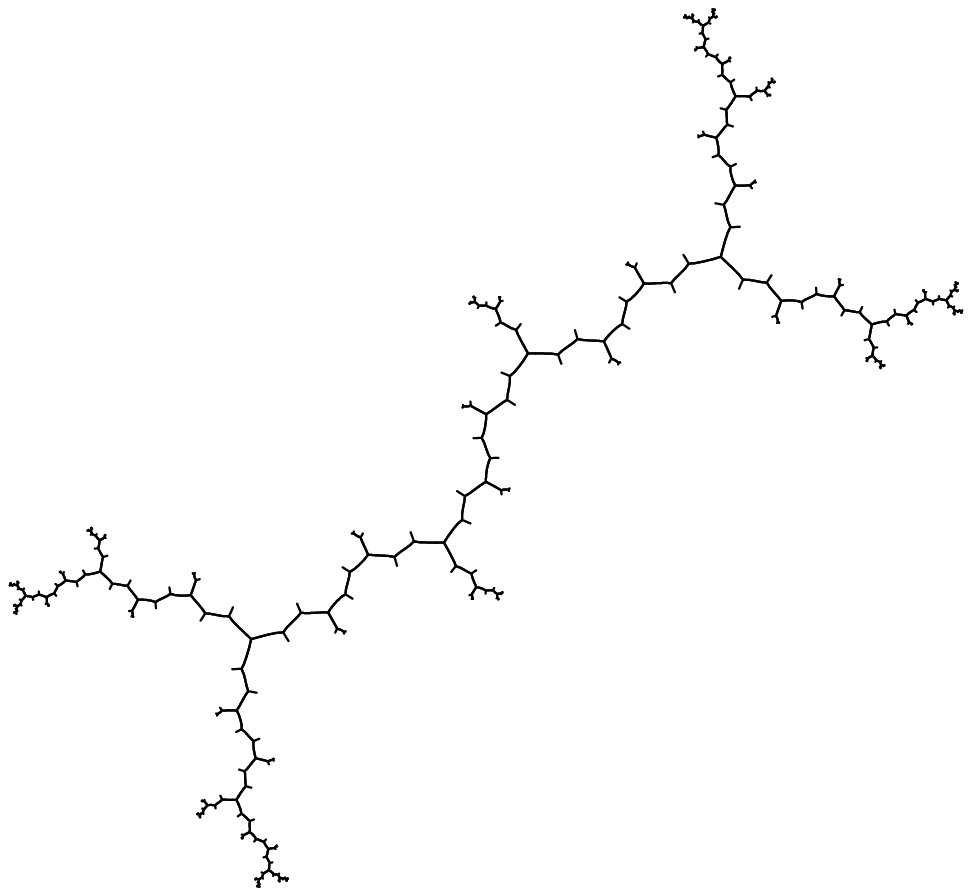}
}
\caption{ \label{Marshall}
Two true trees drawn by Marshall and Rohde's program. The
data on the left was a randomly constructed tree with 
75 edges. 
The   tree edges are approximated by polygons 
so are not highly accurate, 
but the  vertices (i.e., the roots of the 
Shabat polynomials) can be computed to over a 1000 
digits of accuracy.
The tree on the right has 1250 edges  whose 
combinatorics were chosen to match   those 
of the Julia set of a certain quadratic polynomial and 
the resulting true tree accurately matches the shape of 
the Julia set.
}
\end{figure}

The paper \cite{classS} contains a generalization of 
Theorem \ref{poly} from polynomials to entire functions.
Given an infinite tree $T $ in the plane satisfying 
certain bounded geometry conditions, this paper 
gives a construction  of an   entire function with 
only two critical values, so that $f^{-1}([-1,1])$ approximates
$T$ in a precise sense.
 Section 15 of \cite{classS} 
describes how Theorem \ref{poly} in this paper can 
be deduced from  the more
intricate construction in that paper.
Other applications are also given, e.g.,  the construction of 
Belyi functions on certain non-compact surfaces and 
the existence of 
an entire function with bounded singular set that has 
a wandering Fatou component (this is impossible for 
entire functions with finite singular sets by a modification of 
Dennis Sullivan's  ``non-wandering'' argument
for rational functions. See  
\cite{MR1196102},
\cite{MR857196},
\cite{MR819553}).

If we require the harmonic measures for the two sides of a tree
edge to be identical, but don't require all edges to have the
same harmonic measure, we get what is called a minimal
continuum. These sets arise as the continua of minimal 
capacity that connect a given finite set.
Minimal continua 
are studied by Herbert Stahl in \cite{Stahl};
 this authoritative paper 
contains extensive history and references for the topic.
 I thank Alex Eremenko for pointing out the connection 
between balanced trees and minimal continua  to me.

\section{Basic properties of conformally balanced trees}

 In this paper,  a finite  plane tree  $T$ will be a connected compact set
in $\complex$ that does 
not separate the plane and is a union of a finite collection of closed 
Jordan arcs, any two of which are either disjoint or have exactly one 
endpoint in common.  The edges of the tree are the interiors of these
arcs and  the vertices are the endpoints.
We shall say that two finite trees in the plane are  
equivalent if there is homeomorphism of the plane that takes 
one to the other. Note that this is more restrictive than 
saying there is a homeomorphism from one tree to the other
since such a map can swap branches in a way that a planar 
homeomorphism cannot. See Figure \ref{TwoTrees}.

\begin{figure}[htb]
\centerline{
	\includegraphics[height=1in]{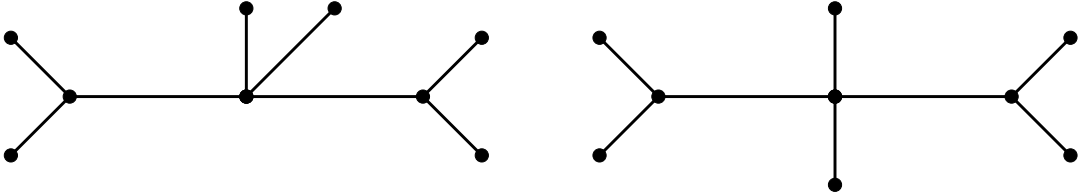}
}
\caption{ \label{TwoTrees}
Two planar trees that are homeomorphic but not equivalent 
(no homeomorphism of the plane maps one to the other).
}
\end{figure}

A planar tree is locally connected, so a conformal map from 
$\disk^* $  to $\Omega = \complex \setminus T$, 
extends continuously to $\circle $.  
We shall always  assume that such a map fixes  $\infty$.

Let $R_n \subset \circle$ be the set of $n$th roots of unity.
A finite tree with $n$ edges  is   conformally balanced   if there is a 
conformal map $ f: \disk^* \to  \Omega = 
\complex \setminus T $ so that each component of 
$\circle \setminus R_{2n}$ is mapped  1-1 onto an edge of  the tree 
and if $I,J$ are two distinct components that map to the same edge, 
then $f^{-1} \circ f$ defines  a length preserving, orientation reversing
 map from one component to 
the other. This expresses precisely the idea that every edge has the 
same harmonic measure and that  harmonic measure on 
each edge is the sum of  harmonic measures corresponding to each side 
separately, and that these two measures are identical.

An orientation preserving 
 homeomorphism $\phi$ of the plane to itself is called 
quasiconformal (or QC for short) if it is absolutely continuous on 
almost all vertical and horizontal lines and satisfies 
$|f_{\overline{z}}|  \leq k |f_z|$ almost  everywhere  for 
some $k < 1$. Such a map is also called $K$-quasiconformal 
where $K = (k+1)/(k-1)$ measures the eccentricity of 
image ellipses  of infinitesimal circles under $f$.
The smallest such $K$ is called the quasiconstant of $f$.
The collection of $K$-quasiconformal maps for a fixed
$K$ form a compact family  with respect to uniform convergence 
on compact sets (assuming the maps are normalized
to fix $\infty$ and two finite points). 
 See Alhfors' book \cite{Ahlfors-QCbook} for 
 this and other  properties of  such  maps. 
The function $ \mu = f_{\overline z}/f_z$ is called 
the dilatation of the map $f$ and the size of $|\mu|$ 
measures how far $f$ is from conformal; if $\mu=0$ 
on an open set, then $f$ is conformal on that set.
There is a composition law for dilatations that implies 
that if  $f$ and $g$ have the 
same dilatation on an open set, then $f^{-1} \circ g$ is conformal
on that set.  
If  $f$ has zero dilatation on the whole plane, then $f$ is a conformal 
linear map, i.e., $f(z) = az+b$. 
A well known quantitative version of this fact is:

\begin{lemma} \label{QC 2} 
Given $K < \infty$ and $\epsilon > 0$ there is a $\delta>0$ so the 
following holds. If $\psi$ is a $K$-quasiconformal map of the plane  
fixing $0,1,\infty$  and if its dilatation is zero except on a 
measure $\delta$  subset of $\disk$, 
then $|f(z) - z| < \epsilon$ for every $z \in \disk$.   
\end{lemma} 

\begin{proof}
One can give more precise estimates, but this version is simply a 
compactness argument. If  $\delta \searrow 0$, then  the maps must 
converge on compact sets to a conformal map fixing $0,1,\infty$, 
i.e., the identity.
\end{proof} 

The measurable Riemann mapping theorem says that 
given any measurable $\mu$ on the plane with $\|\mu\|_\infty 
<1$, there is a quasiconformal map $f$ with dilatation 
$\mu$. This is the key result about quasiconformal maps that 
we need, as illustrated by the following definition and 
lemma.

A tree $T$  is QC-balanced if there is a quasiconformal 
mapping  $\phi: \disk^* \to \Omega$  so that  components 
of $\circle \setminus R_{2n}$ are mapped  to edges of $T$
 and  when two components 
are mapped  to the same edge,  $\phi^{-1} \circ \phi$ is length preserving 
and orientation reversing between the components (this is the same 
the definition of 
 conformally balanced, except that we have replaced the conformal 
map by a quasiconformal map). 

\begin{lemma}  \label{QC 1}
Suppose $T$ is a QC-balanced tree. Then there is a quasiconformal map 
of the plane to itself sending $T$ to a conformally balanced tree.
\end{lemma} 

\begin{proof}
Let $\phi: \disk^* \to \Omega$ be the QC map in the definition of QC-balanced
and let $\mu$ be the dilatation of $\phi^{-1}$ on $\Omega$. By the 
measurable Riemann mapping theorem there is a quasiconformal $\psi$ on the 
plane with the same dilatation and thus $ \psi \circ \phi$ is conformal. 
Hence $\psi(T)$ is conformally balanced.
\end{proof} 

To say this in a slightly different way, if we compose the 
QC map $\phi^{-1}: \Omega \to \disk$ with $z^d$ and the 
Joukowsky map we get a locally QC map $g$  from $\Omega$ to 
$U = \complex \setminus [-1,1] $ that extends continuously
to the whole plane. Then $g$ is a $d$-to-$1$ quasiregular map 
with singular  values $\pm 1$ 
and the measurable Riemann mapping theorem implies there is a 
quasiconformal map $\phi$ so that $f = g \circ \phi^{-1}$ is 
a $d$-to-$1$ holomorphic map  with the same singular values as
$g$, i.e., $f$ is a Shabat polynomial. 

Thus we can construct  a conformally balanced tree by first 
constructing a QC-balanced tree and ``fixing it'' with a QC map.

\begin{lemma}
Every finite tree in the plane can be mapped to a conformally 
balanced tree by a homeomorphism of the plane. 
\end{lemma}

\begin{proof}
By the previous lemma, it suffices to map $T$ to a 
QC-balanced tree.
Every planar tree is  equivalent to one with 
 straight segments for edges and such a tree is clearly   equivalent
to one with smooth edges meeting with equal angles 
at each vertex (i.e., at a degree three vertex the edges meet 
at angle $120^\circ$).  For such a tree the harmonic measures 
for two sides of any edge  decay at the same rate at each 
endpoints (the decay rate may be different at the two endpoints 
of an edge if the endpoints have different degrees)
and this means the harmonic measures
for the two sides of an edge are within a bounded factor of each 
other (depending on the tree and the edge).

Let $E$ be the preimages of the vertices under $f$. If $T$ has 
$n$ edges, there are $2n$ points in $E$. The $2n$ 
components of $\circle \setminus E$  are paired by the relation 
of mapping to the same edge of $T$. Suppose $I,J$ is such a pair.
Then $f^{-1} \circ f: I \to J$ defines 
a biLipschitz map  between such a pair of  corresponding  arcs $I, J$. 
In what follows, $f^{-1} \circ f$ will always refer this this 
type of map (between different intervals), rather than the 
identity from an interval to itself.

Let $L:J \to I$ be the map that multiplies length by a factor 
of $  |I|/|J|$ and reverses orientation. Define $L$ on $I$ to 
be the inverse of this map.  Then $g= L \circ f^{-1} \circ f$ 
maps $I$ to $I$, preserves orientation and is biLipschitz. 
Define $g: J\to J$ to be the identity.  Then define $g$  and $L$ on every other 
pair of edge-arcs  in the same way. The result is a biLipschitz, 
orientation preserving map of the circle to itself so that 
$f(g(x)) = f(L(x))$ for every $x \in \circle$.
Note that $g$ can be extended to a quasiconformal self-map $\phi$
 of $\disk^*$. 
Let $F= g(E)$. Then 
 there is a quasiconformal self-map $h$  of $\disk^*$ that 
maps $E$ to $R_{2n}$ (roots of unity)  and  $|h'|$ is constant on each 
complementary arc.  

Consider the map $ \Phi = f \circ g^{-1}   \circ h^{-1}$. 
It is quasiconformal on $\disk^*$, maps $\circle$ onto $T$, and sends $R_{2n}$ 
to the vertices. If two arcs of $\circle \setminus R_{2n}$ are mapped to 
the same edge, then  $\Phi^{-1}\circ \Phi $ is length preserving.
Hence $T$ is QC-balanced and thus has a QC image that is 
conformally balanced. 
\end{proof}

\begin{lemma} \label{removable}
A conformally balanced tree with $n$ edges
 is of the form $T = p^{-1}([-1,1])$ 
for some polynomial $p$ that has exactly two critical values 
at $ \{ -1,1\}$. The vertices of degree $>1$ of $T$ are exactly the critical 
points of $p$ and the degree equals the order of the zero of $p'$ plus $1$.
  The edges of $T$ are analytic curves. 
\end{lemma} 

\begin{proof}
The proof is essentially given in the introduction. The only step that 
was not justified there was the statement  that $g(f(z)^d))$  ``extends continuously 
to the whole plane and hence is entire and hence a polynomial''.
This requires some proof.

We have already seen that a conformally balanced tree  $T$ is the 
planar  quasiconformal 
image of a  finite tree with smooth edges such that  
all angles at vertices are non-zero.
This means the complement of $T$  is a John 
domain and hence is removable for  $W^{1,2}$ mappings (one derivative 
in $L^2$; a QC map raised to a power is in this class locally). 
 See \cite{MR1315551}, 
\cite{MR1785402}. Thus if $ g(f(z)^d)$ is a continuous function that is holomorphic 
off $T$, then it is entire. This finishes the proof sketched in the 
introduction.
\end{proof} 

\begin{lemma} 
Two equivalent conformally balanced trees are the same 
up to a conformal linear map. 
\end{lemma} 

\begin{proof} 
If two conformally balanced trees have the same topology, then
there is a conformal map between their complements that extends 
continuously to the whole plane.
Since the edges of balanced tree must be analytic, they are removable 
for conformal maps, so the map is conformal everywhere and hence
is linear.
\end{proof} 

\begin{cor}
Every finite  planar tree is  equivalent 
to a  conformally balanced tree that is unique up to linear maps.  
\end{cor}

In particular, the number of conformally balanced trees with $n$ 
vertices  (up to linear equivalence) is the same as the number of
plane trees with $n$ vertices.
These  can be counted using  P{\'o}lya's 
enumeration method as in \cite{MR0166776}, \cite{MR0376411}.


\section{The construction on $\circle$}

The proof of Theorem \ref{Main}  consists of constructing a tree  $T$
approximating $K$,  pre-composing  the conformal 
map $f: \disk^* \to \Omega = \complex \setminus T$ by a QC self-map  $\phi$
of $\disk$ and finally post-composing $f$ by a QC map $\psi$ of $\Omega $ onto 
$\Omega' = \complex \setminus T'$ 
  where $T'$ is a QC-balanced  tree containing $T$ 
and is close to it in the Hausdorff metric.
The QC map $\psi \circ f \circ \phi$  associated to $T'$ 
will have uniformly bounded dilation  and the  support
of the dilatation  will have  as small area  
as we wish, so invoking the measurable Riemann mapping theorem 
and Lemma \ref{QC 2} 
gives a conformally balanced tree that approximates $K$.

In this section we construct $T$ and
 the pre-composition map $\phi$ of $\disk^*$.
The tree $T'$ and the QC map $\psi:\Omega \to \Omega'$ will be constructed
in the next section.

Suppose $K$ is a compact connected set.   Choose a large 
integer $D$ and let ${\cal C}$ be the collection of dyadic 
square of size $2^{-D}$ that hit $K$. The corners and edges 
of these squares form a finite graph in the plane and we 
take a spanning tree for this graph. Then add segments of 
length $\frac 14 2^{-D}$ to any vertices of degree $< 4$  
so that every vertex in the resulting tree $T$ has 
degree $1$ or $4$, every edge is still vertical or 
horizontal and every edge has length either $2^{-D}$ or 
$2^{-D-2}$. See Figure \ref{the tree}.
The tree $T$ approximates $K$ to within $2^{-D+1}$  in the 
Hausdorff metric. 

\begin{figure}[htb]
\centerline{
	\includegraphics[height=1.3in]{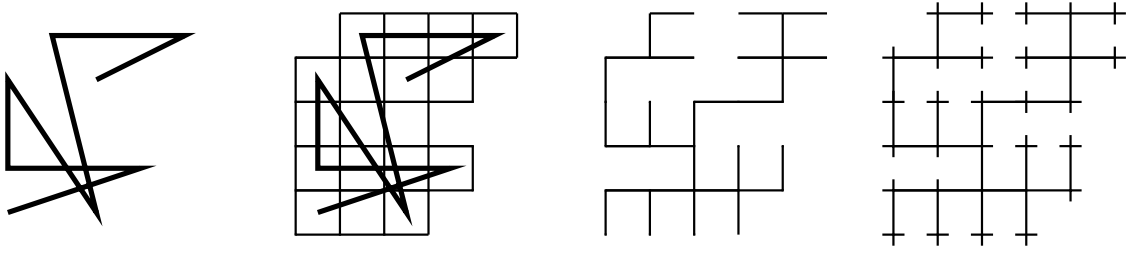} 
}
\caption{ \label{the tree}
A continua is covered by dyadic boxes and an approximated 
tree is formed from the boxes' edges. Some extra segments 
are added to make every degree $1$ or $4$ and every edge
have length $2^{-D}$ or $2^{-D-2}$. 
}
\end{figure}

Why did we add the extra segments to make every degree $1$ 
or $4$? This is more of a convenience than a necessity.
The condition insures that for any edge, the 
 harmonic measures   for the two sides  have the 
same behavior as we approach an endpoint, i.e., 
$\frac {d \omega_1}{d \omega_2}$ is bounded above and 
below on the whole edge (in fact, this function extends to 
be analytic on a neighborhood of the edge).   The precise
version of this fact that we will use is: 

\begin{lemma}
Suppose $e$ is an open edge of $T$, $f:\disk^* \to \Omega$ is a 
conformal map  onto the exterior of $T$ and $I, J \subset \circle$
are the two components of $f^{-1}(e)$. The map $ g=f^{-1} \circ f$ 
defined from $I$ to $J$ has an extension to a conformal map from 
a neighborhood $\Omega_I$  of $I$ to a neighborhood  $\Omega_J$ of $J$.
Moreover, 
$$ \dist(I,  \partial \Omega_I) \geq C_1 |I|,$$
for some absolute $ C_1 >0$. The same estimate holds for $J$ and $\Omega_J$.
Also, 
$$ C_2^{-1} \leq |g'| \frac {|I|}{|J|} \leq C_2$$
on $I$  for some absolute $C_2 < \infty$.
\end{lemma} 

\begin{proof}
This is just an application of the Schwarz reflection principle.
 We first 
consider the case when the endpoints of $e$ both have 
degree $4$, as in  Figure \ref{Reflect}.
 Let $e'$ be the edge $e$ with  perpendicular 
segments of length $\frac 14 |e|$ added at either end, 
so as to bound three sides of a rectangle $R$, whose 
preimage under $f$ is an open set $\Omega_I^+$ in $\disk^*$ with 
$I$ in its boundary. This open set, together with its
boundary $I'$  on $\circle$ and its reflection across 
$\circle$ will be the set $\Omega_I$.

Map $\Omega_I^+$ to $R$ by $f$, follow by a reflection 
across $e$ to another rectangle, map this to a set 
$\Omega_J^+$ by $f^{-1}$ and reflect this across 
$\circle $ to the set $\Omega_J^-$. Let $\Omega_J$ 
be the  union of
$\Omega_J^+, \Omega_J^-$ and the interior (in $\circle$) of
 their common boundary $J'$. 

This composition is made up of two conformal maps and 
two reflections, so is a conformal map $ \Omega_I^+ \to \Omega_J^-$
  and  sends  $I'$ to $J'$, 
so by the Schwarz reflection principle,  it  extends to be 
a conformal map from $\Omega_I$ to $\Omega_J$.

Clearly the harmonic measure of $e$ in $\Omega = \complex
\setminus T$ from any point of the opposite side of $R$ 
is bounded uniformly away from one, so the
same is true of $I$ in $\disk^*$ from any point of 
$\partial \Omega_I \cap \disk^*$. This implies $\partial
\Omega_I$ is at least distance $C_1 |I|$ from $I$ for 
some absolute $C_1$. The same applies to $J$ and 
$\Omega_J$.  The Koebe $\frac 14$-theorem now implies 
that $g$ has derivative comparable to $|J|/|I|$ on $I$, 
again with absolute constants. 
\begin{figure}[htb]
\centerline{
\includegraphics[height=1.83in]{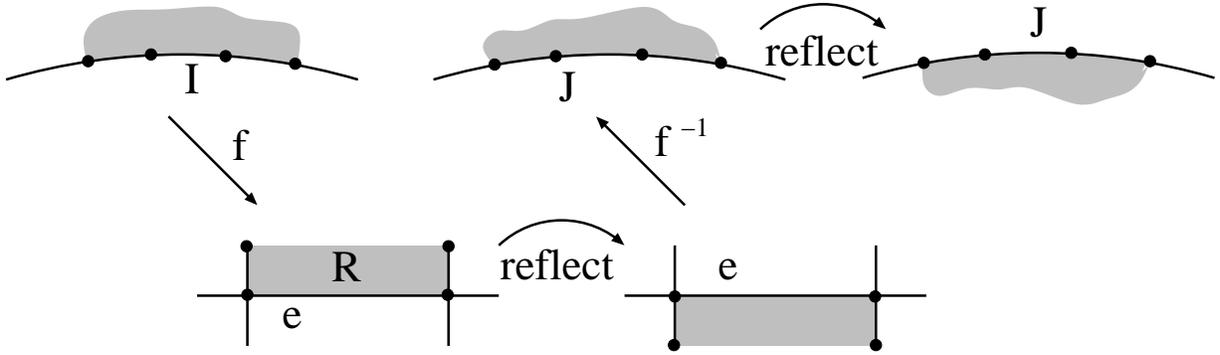} 
}
\caption{ \label{Reflect}
Proof that $ g= f^{-1} \circ f$ has a conformal extension to a 
neighborhood of $I$ if the image connects two vertices of degree $4$. 
}
\end{figure}

If $e$ has one vertex of degree $1$ and the other of 
degree $4$, the argument is very similar. In this case, 
the intervals $I$ and $J$ are adjacent and we take $R$ 
as shown in Figure \ref{Reflect2}. Its preimage under 
$f$ is the light gray region above the circle   that
we will denote $\Omega_{IJ}^+$, and the darker region below 
the circle is its reflection $\Omega_{IJ}^-$. As before, the composition 
of the four maps is conformal  between these domains, and 
hence it has a conformal extension from the obvious 
domain  $\Omega_{IJ}$ to itself.
The remaining conclusions follow just as before. 

\begin{figure}[htb]
\centerline{
\includegraphics[height=2.0in]{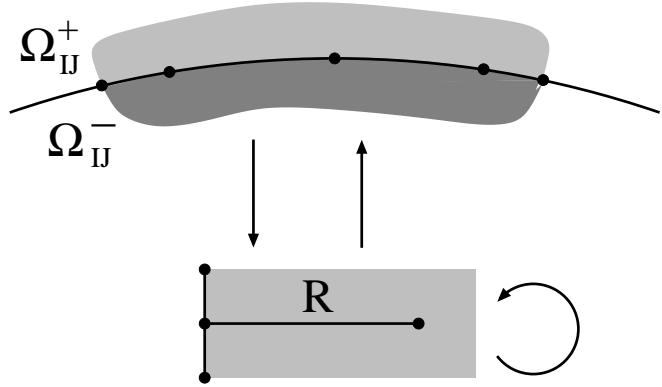} 
}
\caption{ \label{Reflect2}
Proof that $ g= f^{-1} \circ f$ has a conformal extension to a 
neighborhood of $I$ if the image connects   a degree $1$ and  degree $4$
vertex. 
}
\end{figure}  
\end{proof}

So the restriction of the mapping $g = f^{-1} \circ f$ 
to  any component of 
$\circle \setminus E$ has a conformal extension to a uniformly 
larger neighborhood (recall $E$ are the preimages under $f$ 
of the vertices of $T$), although the map itself may have jump 
discontinuities at the points of $E$. This is not quite the 
same as  ``piecewise 
analytic'' since this term usually includes continuity at 
the endpoints.  We want to  approximate  $g$ 
by a piecewise linear map on each component of $\circle \setminus E$ 
by adding more points 
into the gaps between $E$ and linearly interpolating the 
values of $g$ between these points. 

\begin{lemma} \label{intervals}  
Suppose $g$  is as above.
Then there is a 
quasiconformal map $\phi$ of $\disk^*$ to itself, and a finite
set $F \subset \circle $ so that $\phi(F)$  contains  $E$ and 
so that $ \phi^{-1}\circ g \circ \phi$ is piecewise linear on each component 
of $\circle \setminus F$.
 The quasiconstant of $\phi$ 
is  uniformly bounded and the dilatation  $\mu$ of $\phi$  can be  
chosen to be supported in any neighborhood of $\circle$ that we want 
(depending on our choice of $F$).
 We let ${\cal I}$ denote the
connected components of $\circle \setminus F$. 
The length of each interval in ${\cal I}$  may be chosen to
be of the form $ 2 \pi 2^{-n}$ for some  integer $n$ (possibly 
different $n$'s for different intervals), 
the lengths  of  adjacent intervals are
within a factor of $2$ of each other and  every 
interval has the same length as at least  one of 
its two neighbors.
If two intervals in ${ \cal I}$ are adjacent and their common 
endpoint  is mapped by $f$ to a vertex of $T$, then they have the same length.
 \end{lemma}

\begin{proof}
Consider a pair $I,J$  of components of $\circle \setminus E$ that 
map to the same edge of $T$.
Subdivide  $I$ and on each subinterval,  let $\phi$ 
be defined as  $g$ followed
by the linear map from $J=g(I)$ back to $I$ that  inverts  $g$ at the 
endpoints. On the interval $ J=g(I)$, $ \phi$ is defined to be the 
identity.  Since $g$ is smooth,  $\phi$ is biLipschitz 
with constant as close to $1$ as we want if the subdivision of $I$
is fine enough.  We can therefore
extend it to a quasiconformal map of the Carleson region 
$$ Q_I = \{ z \in \disk^*: z/|z| \in I, |z|-1 < |I| \},$$
to itself that is the identity on $\partial Q_I \setminus I$.
Define $\phi$ on the rest of $\disk^*$ as the identity and
define it in $\disk$ by reflection.  This map has the desired
piecewise linear property, but we still need to adjust the sizes of 
the intervals. 

 To make 
adjacent intervals have comparable length with a factor of $2$,
 we simply split the 
larger in 2 equal pieces whenever this fails; the 
shortest interval will never be split and a shorter interval 
will never be produced, so the process ends 
after a finite number of steps. 

To make notation easier, we normalize arclength on the circle 
to be $1$.
 To make sure that the normalized interval lengths are  powers of
$2$, cover the circle by disjoint dyadic intervals
that are at most $1/4$ as long as any of the  intervals from 
the collection that  they hit,   
and that are  maximal with respect to this property. 
Such a dyadic interval has at least $\frac 18$th of the length 
of the shortest interval  it hits, and is contained in the 
union of this interval and one of its neighbors, which   is 
at most twice as long. Thus each of our dyadic intervals 
has length between $\frac 14 $ and  $\frac 1{16}$ times 
the  length of  any interval in our collection  that it  
intersects.

If we replace each interval $I$  in ${\cal I}$ by the  union of 
dyadic intervals  
in ${\cal D}$  that  are contained in  $I$  or contain $I$'s left 
endpoint,  the new interval $I'$ has comparable length and 
is a union of between $4$ and $16$ dyadic intervals. By splitting 
some of the dyadic intervals in two, we can insure it is always 
a union of $16$ dyadic intervals.

If necessary, we can repeat the ``split the larger neighbor''
argument to insure adjacent intervals have lengths within a factor of 
$2$ of each other. We end by splitting every interval
into four equal subintervals 
to make sure every interval has at least one equal sized neighbor.
If $I$ and $J$ are both  adjacent to a point mapping to a vertex, 
 but are  not of equal length, then one is exactly 
twice as long as the other. Subdivide the longer one and the 
adjacent interval of the same length. Then 
the two segments adjacent to the vertex preimage are equal and 
all the intervals  still satisfy all the other requirements.
This final collection is the desired collection ${\cal I}$.
\end{proof} 

\section{The construction on $T$}

In this section, we define a  tree $T'$ containing $T$ and 
a series of quasiconformal maps 
$$  \complex \setminus T = \Omega \to \Omega_0 \to 
    \Omega_1 \to \Omega_2 \to \Omega_3 \to 
     \Omega_4 = \complex \setminus T'. $$
If we denote the composition by $\psi$, then our construction will 
have the property that $T'$ is a QC-balanced tree via the 
map  $\psi \circ f \circ \phi: \disk^* \to \complex \setminus T'$. 
Moreover, the dilatation of this map will be uniformly bounded and 
the support of its dilation is mapped into as small a neighborhood 
of $T$ as we wish (equivalently, the inverse
map, which automatically has the same quasiconstant, has dilatation 
supported in an arbitrarily small neighborhood of $T$).
  Thus using Lemmas \ref{QC 2} and \ref{QC 1}
will yield a conformally 
balanced tree that approximates $T$, and hence $K$.

To simplify, we will rescale $T$ to correspond 
to a unit grid (i.e., take $D=0$).

In order to draw simpler pictures, we want to avoid the corners
in $T$ created by the vertices of degree $4$. The first map 
$\psi_0: \Omega \to \Omega_0 \subset \Omega$ simply pulls the domain 
way from these corners in a  uniformly QC way.  
Choose $ 0< \delta \ll 1$ to be a small power of $2$ (how small 
will be determined during the course of the construction) and 
for each degree $4$ vertex in $T$ remove  the four 
 $\delta \times \delta$ subsquares  of $\Omega$ that have 
this vertex as a corner.  This gives $\Omega_0$. 
Let $\Omega'$ be $\Omega$ with slits of length $\sqrt{2} \delta$ 
bisecting each corner of $\Omega$  removed (these are diagonals of 
the squares we just removed).  
There is a uniformly quasiconformal map $\psi_0: \Omega' 
\to \Omega_0$ that is affine on each edge and equals the 
identity outside a $\delta$-neighborhood of $T$. See 
Figure \ref{Omega0}. (Note that $\psi_0$ is not quasiconformal  on 
$\Omega$ because it is not continuous along the slits 
defining $\Omega'$, but we will finish the construction 
by composing with $\psi_0^{-1}$ to ``fill in'' the corners 
and the composed map will have a continuous, quasiconformal 
extension to all of $\Omega$.)

\begin{figure}[htb]
\centerline{
	\includegraphics[height=2in]{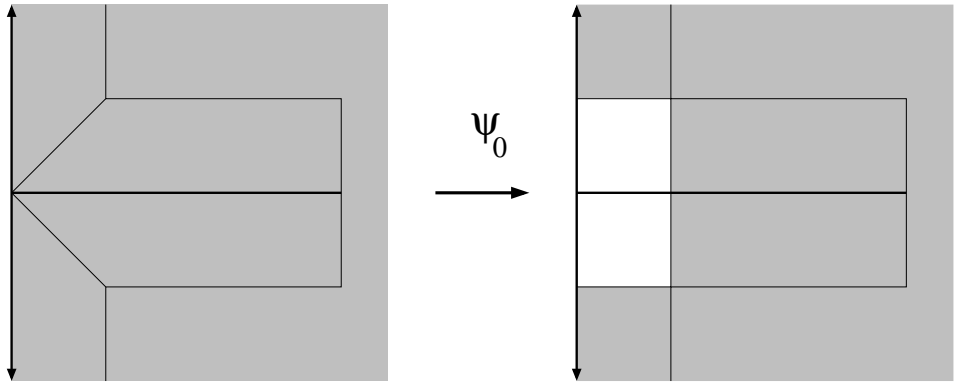}
}
\caption{ \label{Omega0}
The map   $\psi_0 : \Omega \to \Omega_0$. It pulls the domain 
away from the corners.
}
\end{figure}

Now we define $\Omega_1$ as the set of points  $z \in \Omega_0$ such that
$$ \dist(z, T) >  \delta \quad \text { or } \quad \dist(z, T) > \sqrt{2} \dist(z, V_1),$$
where $V_1$ is the finite set of degree $1$ vertices of $T$.
Thus $\Omega_1$ is a polygon where most of the edges are parallel
to edges of $T$, except in a neighborhood of each degree one
vertex where the boundary slopes down to hit $T$ at the vertex.
We can clearly map $\Omega_0 \to  \Omega_1$ by a uniformly QC
map with dilatation supported in a $\delta$-neighborhood of $T$.
See  Figure \ref{Omega1}. 
If $\delta$ is small enough then any  interval of length $\delta$
with 
one endpoint at a degree $1$ vertex of $T$ is contained in the 
image of the two ${\cal I}$ intervals on the circle that 
are  adjacent to that vertex. 
Assume $\delta$ has been chosen small enough to make this happen 
at every degree $1$ vertex.

\begin{figure}[htb]
\centerline{
	\includegraphics[height=2in]{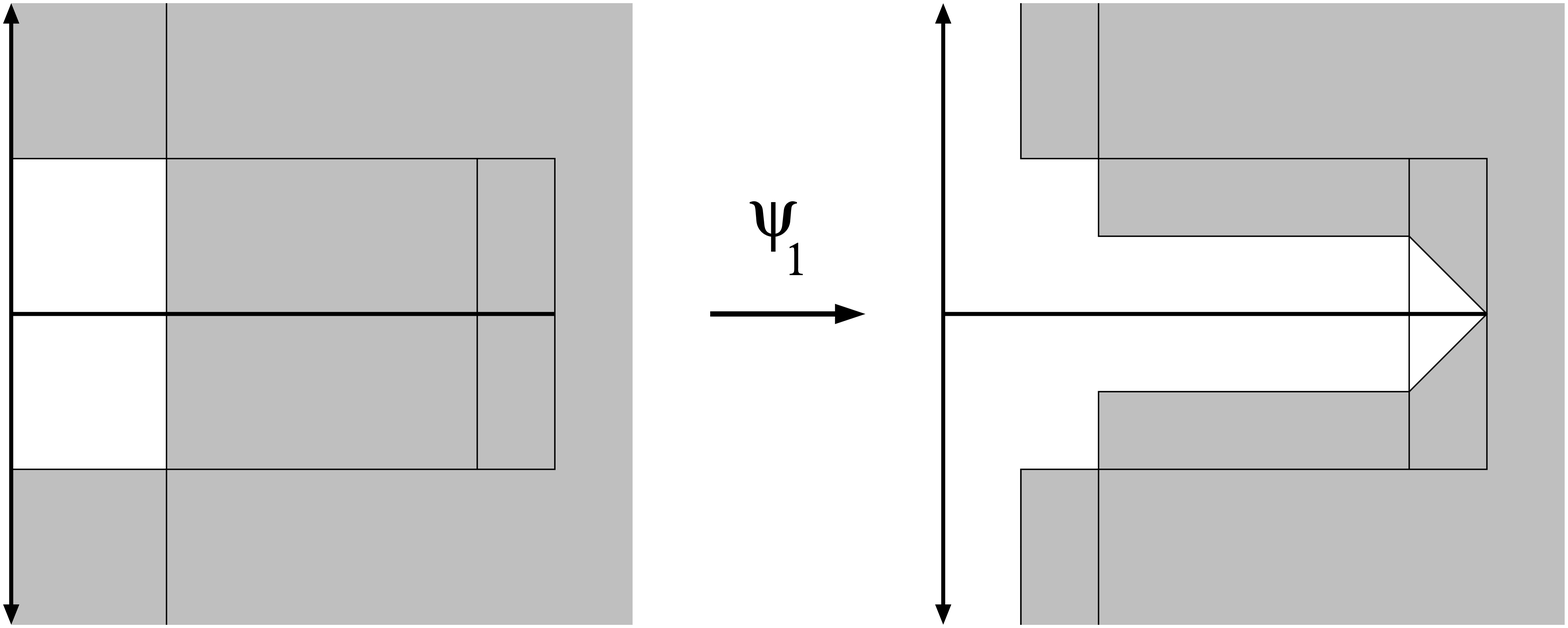}
}
\caption{ \label{Omega1}
The map $\psi_1 : \Omega_0 \to \Omega_1$.
}
\end{figure}

Let ${\cal J}$ denote the segments in $T$ that are of the 
form $\partial \Omega_0 \cap e$ for some edge $e$ of $T$. 
Each edge $e$ of $T$ either connects two vertices of 
degree four or connects a vertex of degree four to a vertex 
of degree one. In the first case, segments  in ${\cal J}$
consist of $e$ with two intervals of length $\delta$ removed 
(one at each endpoint), and in the second case we only remove 
an interval at the degree four vertex. 

  The map $\psi_0 \circ f \circ \phi^{-1}$ sends 
each element of ${\cal I}$ into some element of ${\cal J}$.
Since each element of ${\cal I}$ has measure that is a power 
of $2$, there is a smallest and largest power that occur and 
we denote these by $2^{-n}$ and $2^{N-n}$.
Then the measure of each element that occurs can
 be written as $2^m 2^{-n-N}$ where 
$   N \leq m \leq 2N$. By taking $N$ larger, if necessary, we can 
assume $2^{-n-N } $ evenly divides $1-2 \delta$ and $\frac 14 -\delta$ 
(the two possible lengths of edges in ${\cal J}$). Thus each element of 
${\cal J}$ can be divided into an integer number of    
disjoint sub-segments of length $ 2^{-n-N} $. This collection 
of subintervals is called ${\cal K}$. 
Taking $N$ larger, if necessary, we may assume $ 2 N 2^{-n-N} < \delta$.

 Each element  $K \in {\cal K}$  is 
associated to two elements of ${\cal I}$ whose images contain  $K$ and 
that correspond to the two sides of $K$.
If the measure of one of these intervals is $2^m 2^{-n-N}$ we 
call $m$ one of the two ``heights'' associated to $K$. Each height 
is associated to one side of $K$.
Lemma \ref{intervals} implies that the heights of intervals in $K$
do not change very quickly. In fact, that lemma implies the 
following facts about intervals in ${\cal K}$: 
\begin{enumerate}
\item adjacent intervals have heights differing by at most $1$,
\item every interval has the same height as at least one of its neighboring 
      intervals, 
\item given a degree $1$ vertex $v$ of $T$, every interval within distance 
     $\delta$ of $v$ has the same height,
\item given one of the $\delta \times \delta$ squares removed from $\Omega$ to 
      form $\Omega_0$, the two intervals adjacent to that square have the same 
      height.  They also have the same heights as the neighboring intervals 
      that are not adjacent to the removed square.
\end{enumerate}

Next we build $\Omega_2$.
For each segment $K \in {\cal K}$ in $\partial \Omega_0$
 we add a rectangle or trapezoid to 
both sides as follows. First suppose $K =[a,b]$ is within distance
$\delta$ of a degree $1$ vertex $v$. This means that 
the heights of $K$ for either side are the same by 
Lemma \ref{intervals} if $\delta$ has been chosen small enough (since intervals 
adjacent to a vertex of the tree have equal measure). 
If  $m$ denotes the height associated to $K$, and if 
 $$ m |K| \geq  \dist(K,v_1) +|K|.$$
then we add a rectangle of size $|K| \times m |K|$ to both sides of $K$. 
Otherwise we add a trapezoid with one side $K$, two sides 
perpendicular to $K$ and the fourth side on $\partial \Omega_1$.
See Figure \ref{Omega_2}. 

\begin{figure}[htb]
\centerline{
	\includegraphics[height=2.5in]{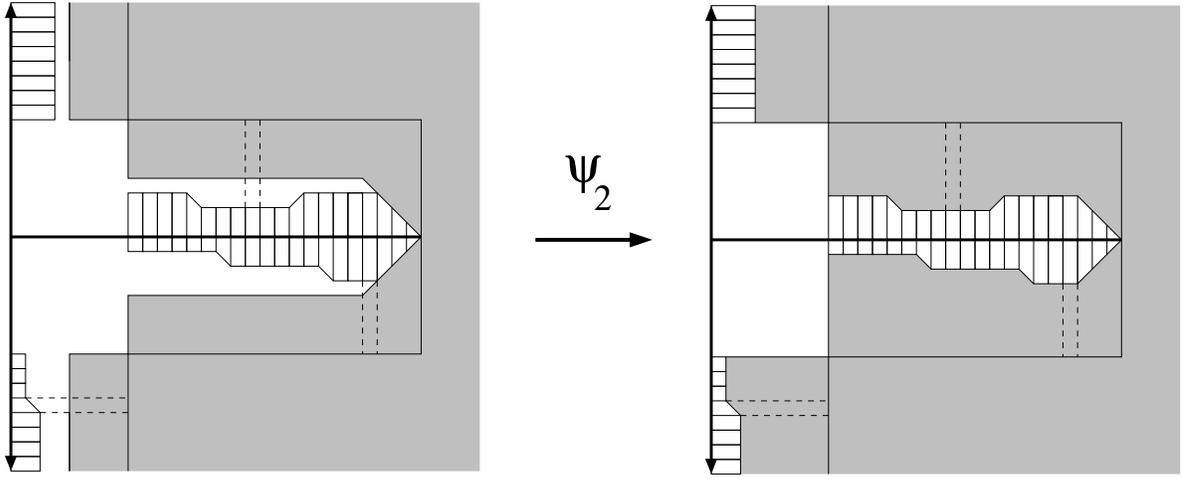}
}
\caption{ \label{Omega_2}
We map $\psi_2:\Omega_1 \to \Omega_2$ by pushing the boundary 
back towards $T$.  The map is the identity for points 
more than $ \delta$ from $T$.
}
\end{figure}

If $K$  is more than distance $\delta$ from any degree one vertex
then consider one side of  $K$ and the two adjacent intervals. If all 
three intervals have the same height $m$, then we add a  $|K| \times m|K|$ 
rectangle with $K$ as one side. Otherwise, one of the adjacent
intervals has the same height $m$ as $K$ and the other has height 
$m^*$ differing by $1$. We add a trapezoid with  base $K$, and two 
parallel sides that are perpendicular to $K$   with side 
lengths  $m|K|$ and $m^* |K|$. The fourth side of the trapezoid is opposite 
$K$ and has length $\sqrt{2} |K|$.

If $K$ has only one neighbor, it must be adjacent to one of the 
removed  ``corner squares'' of $\Omega_0$. 
As noted earlier, it must have the same height as 
its immediate neighbor, as well as the other  interval of ${\cal K}$ 
 adjacent to the same corner  square.

Let $W$ be the union of all these closed rectangles and trapezoids, 
together with the closures of  $ \delta \times  \delta$ squares removed
from $\Omega$ to form $\Omega_0$. The union is a closed connected set 
and the complement is the open set $\Omega_2$.  Clearly $\Omega_1$ 
can be mapped to $\Omega_2$ by a quasiconformal map $\psi_2$  that
is piecewise affine, has uniformly bounded quasiconstant and has 
dilatation supported in a $\delta$-neighborhood of $T$. 
See Figure \ref{Omega_2}. 

If we add all the  open rectangles and trapezoids to $\Omega_1$, 
along with their edge on  $\partial \Omega_2$,
 we get an open set $\Omega_3$ containing 
$\Omega_2$.  We define $\Omega_4 = \psi_0^{-1} (\Omega_3)$ and 
$T' = \partial \Omega_r$. Clearly this is a linear tree
that contains $T$.
See Figure \ref{FillCorners2}.

\begin{figure}[htb]
\centerline{
	\includegraphics[height=1.5in]{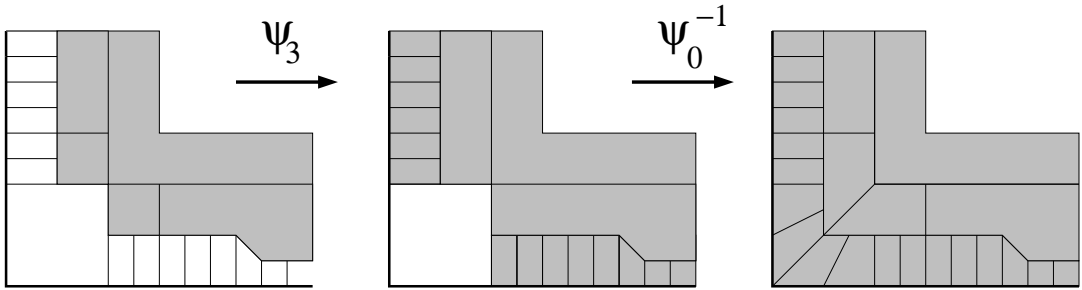}
}
\caption{ \label{FillCorners2}
The map  $\psi_3: \Omega_2 \to \Omega_3$ fills in rectangles and trapezoids 
and then  $\psi_0^{-1} $  ``refills'' the corners.  The 
composition $ \psi = \psi_0^{-1} \circ \psi_3 \circ \psi_2 \circ \psi_1 \circ \psi_0$
has uniformly bounded dilatation, is the identity outside a $\delta$-neighborhood 
of $T$ and  has a continuous extension $\Omega \to \Omega_4$.
}
\end{figure}

The only object not yet defined is the quasiconformal map 
$\psi_3 : \Omega_2 \to \Omega_3$. Again, the map is the 
identity far from $T$, and each connected component of 
$\Omega_3 \setminus \Omega_2$ is a rectangle or a trapezoid 
(we will denote either type of region by $R$) 
 and is  the image under $\psi_3$ 
 of a square in $\Omega_2$ that shares  a side  with 
$R$. See Figure \ref{FillTubes}.
 There are three types of maps to describe: $m$-rectangle maps, 
$m$-trapezoid maps and $m$-tip maps.
Each of these maps takes a region in $\Omega_2$ (either a triangle 
or square) with one boundary segment    $I$ on $\partial \Omega_2$
 and expands it into the component of 
 $\Omega_3 \setminus \Omega_2$ attached along $I$ (this component is 
either a rectangle, a trapezoid or a triangle). See Figure 
\ref{FillTubes}. Each map is the identity on 
$\partial R \cap \Omega_2$, so the map can be extended as 
the identity to the rest of the plane.
 We will describe each type of map separately.

\begin{figure}[htb]
\centerline{
\includegraphics[height=3.4in]{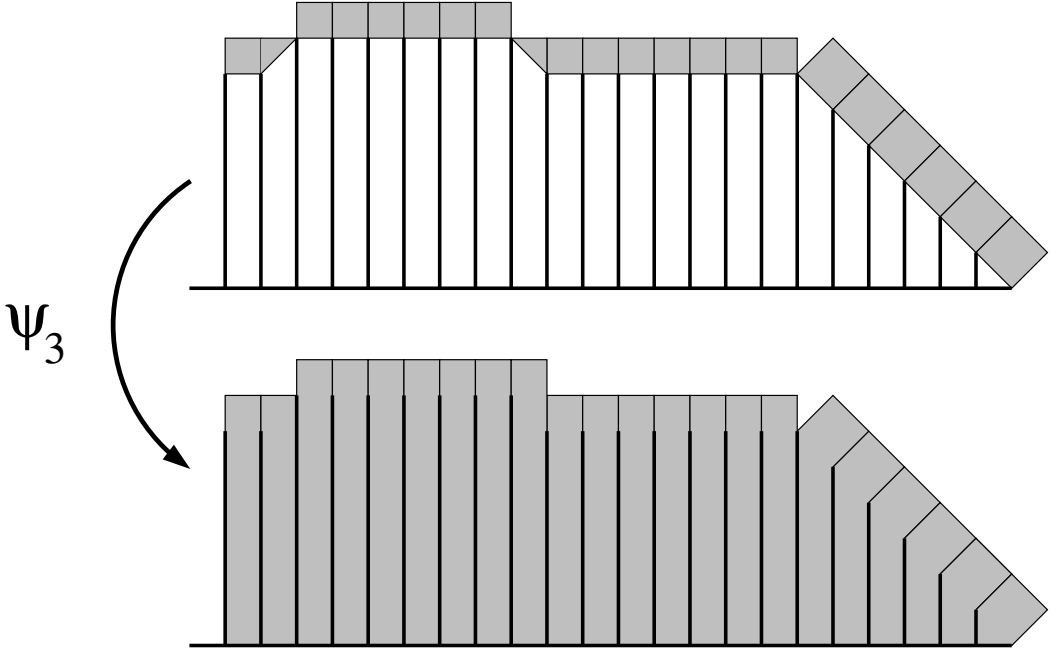}
}
\caption{ \label{FillTubes}
The map $\psi_3 : \Omega_2 \to \Omega_3$ is made up of three 
types of maps that expand a square or triangle in $\Omega_2$ 
into  a rectangle or trapezoid in 
$\Omega_3 \setminus \Omega_2$. 
}
\end{figure}
\begin{figure}[htb]
\centerline{
\includegraphics[height=2in]{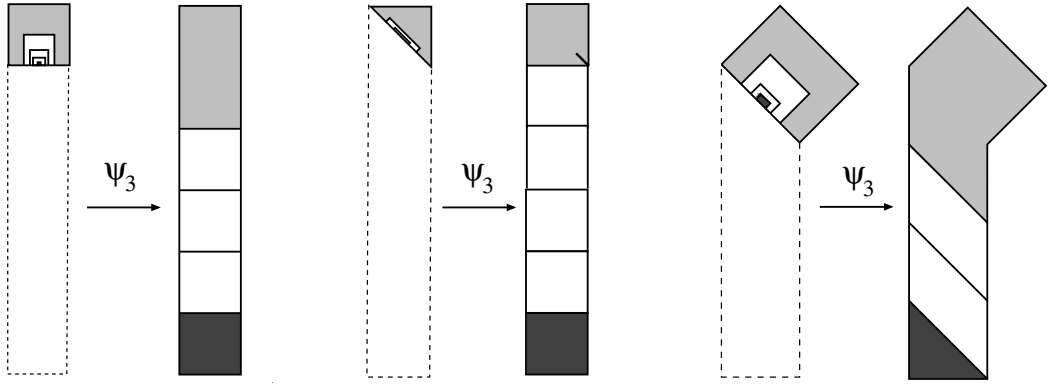}
}
\caption{ \label{MapTypes}
This shows in more detail how the map in Figure \ref{FillTubes} 
expands $\Omega_2$ into $\Omega_3$. In each case the domain 
is cut into decreasing, nested pieces and the pieces are
expanded to shapes that fill the image. The boundary expansion on 
the $k$th piece is $2^k$ in a sense that is made precise 
in the text. The details of each type of map are shown in 
Figures \ref{RectMap}, \ref{TrapMap} and \ref{TipMap}. 
}
\end{figure}

{\bf Rectangle maps:} 
An $m$-rectangle map sends a unit square $S$ to a  $ 1 \times m$ rectangle
$R$. 
We write $R$ as a union of $m$ adjacent unit squares $R  = \cup_{k=1^m} S_k$ 
with  $S_1 = S$.
The boundary values of the map are  as follows.  The map is the 
identity on $\partial S \cap \partial R$ (this is three sides of the 
square) and the fourth side is mapped to the rest of $R$.
starting at the endpoints,  divide the fourth side symmetrically 
 into  two intervals of lengths $4^{-k}$ for $k=1, \dots , m$ (the 
longest adjacent to the endpoints, the shortest adjacent to the midpoint).
For $k = 1, \dots m$  intervals of length $4^{-k}$ are mapped 
affinely to  the part of $\partial S_k$ on the long sides of $R$. 
The  union of the two intervals of length $2^{-m}$ is  mapped affinely 
to the short side of $R$ on $\partial S_k$.  That these 
boundary values can be attained by a uniformly quasiconformal map 
is apparent from the diagrams in  Figures   \ref{MapTypes} and \ref{RectMap}. 

The first figure 
shows how to subdivide the square into $m-1$ nested polygonal regions $P_1, 
\dots , P_{m-1}$;   $P_1$ 
maps to a $1 \times 2$ sub-rectangle in $R$, $P_2, \dots ,P_{m-2}$ are all similar
to each other 
and map to squares,  and $P_{m-1}$ is a square mapping to a square (but not in the 
obvious way, since one of its sides must map to three sides in the image).
These three maps  are  constructed 
in Figure \ref{RectMap}  by showing compatible triangulations for domains 
and ranges (i.e., the triangulations are in a  1-1 correspondence that 
preserves adjacency). Given compatible triangulations of two regions we 
can define a quasiconformal map between them by taking the obvious piecewise 
affine maps between triangles.  It is now an easy exercise to check that the 
mappings induced by the triangulations have the boundary values described 
above.

\begin{figure}[htb]
\centerline{
\includegraphics[height=2.1in]{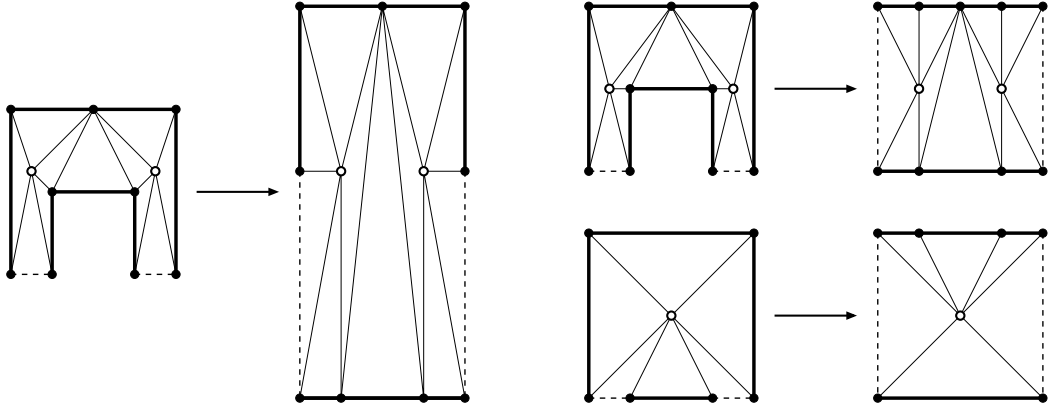}
}
\caption{ \label{RectMap}
The $m$-rectangle map sends a square to a $1 \times m$  rectangle. The 
map is composed from three types of pieces: one each 
for mapping onto  the ``top''
box  (lighter shading in Figure \ref{MapTypes})  and
``bottom'' (darker)  and another that is repeated in all
 the ``middle'' boxes  (white). The triangulations 
define piecewise affine maps between the polygons.
}
\end{figure}

{\bf Trapezoid maps:} 
A $m$-trapezoid map also maps into  a $1 \times m$ rectangle
 $R$ as above, but the domain  of this map
is now a right triangle that we may identify with one half of the 
top square cut by a diagonal. The boundary map is the identity 
on the legs of this triangle. There is an asymmetry to the construction 
and we assume the picture is as shown in Figure \ref{MapTypes}, so that
the domain of the map is the upper right half of the top square.
 The hypotenuse  of the triangle is divided into pairs  of  intervals 
of size $ 4^{-k} $, $k=1, \dots$ as before and the left half is mapped 
to the left side  of the rectangle as before. On the right side the 
rightmost  interval has length $\frac 14$ and is folded onto itself to 
form a slit of length $1/8$  in the rectangle; this slit is not in the 
image of the interior. The remaining smaller intervals are mapped to 
the right side of the rectangle just as before. Figure \ref{TrapMap} shows 
how to divide the triangle into regions and map these regions into the 
rectangle. We only show the details for the top piece; the  lower pieces
are affinely stretched to be similar to the rectangle map pieces  and 
then mapped exactly as in the rectangle maps. 

\begin{figure}[htb]
\centerline{
\includegraphics[height=1.5in]{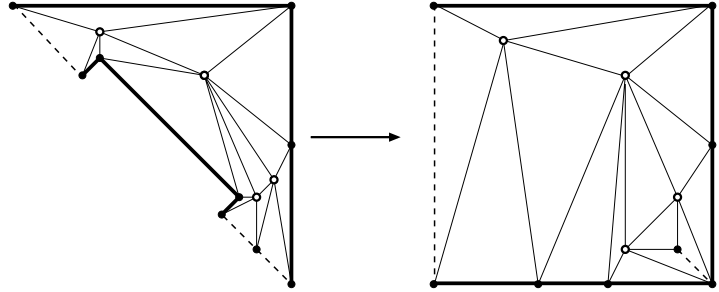}
}
\caption{ \label{TrapMap}
The $m$-trapezoid map sends the triangle to a $1 \times m $
rectangle. The middle and bottom maps are the same, after an 
affine stretching, to the middle and bottom maps of the 
rectangle map, so we only show the construction for the 
top piece. Note that the image is not a simple polygon; 
a piece of the boundary is folded onto itself to form a 
slit in the image. This is necessary for the trapezoid map 
to match rectangle maps of different heights on either side. 
}
\end{figure}

Note that $m$-trapezoid maps interpolate between $m$-rectangle  maps and 
$(m-1)$-rectangle maps. The boundary segments of $\partial \Omega_2$ corresponding 
to each map have measure $2^m 2^{-n-M}$. Dividing these segments into 
$2^{m+1}$ equal, disjoint subsegments and applying the ``filling map'' partitions 
the sides and bottom of the rectangular image into intervals. Whenever two rectangles 
or trapezoids share a side, we want the partitions of these sides to be 
identical. Each piece of the partition is one edge of our QC-balanced tree and
we want them to have equal measure. Obviously two rectangles maps of the same 
height match up and the definition of the $m$-trapezoid map  is designed so that 
it matches a $m$-rectangle map on one side and a $(m-1$)-rectangle map on the other.
The top side of an $(m-1)$-rectangle map has half the measure of the top 
of a $m$-rectangle, so the trapezoid map matches up intervals of equal 
measure.

{\bf Tip maps:}
The third type of map is the $m$-tip map. The details are described 
in Figure \ref{TipMap}. Each $m$-tip map is designed to match a 
$(m+1)$-tip map (or a $m$-rectangle map) on its longer vertical side and a 
$(m-1)$-tip map on its shorter vertical side. Once again the boundary map 
is the identity on the top three sides of the domain, and maps the 
bottom side to the sides and bottom of the image trapezoid. The bottom 
side of the domain square is again divided into symmetric 
pairs of  intervals of length 
$  4^{-m}$. The leftmost and rightmost are mapped to the unit 
segments of the trapezoids vertical sides, but since these 
sides are different lengths, the images are displaced vertically 
with respect to each other, so they form the vertical sides 
of a parallelogram. Intervals of length $4^{-k}$ are mapped to 
vertical sides of the lower parallelograms. After $m$ steps, 
the parallelogram hits the bottom edge and the rest of the 
domain square is mapped to the bottom triangle as illustrated in 
Figure \ref{TipMap}. The last map, adjacent to the tip, is a 
special case that is illustrated in Figure \ref{TipMap}. 

\begin{figure}[htb]
\centerline{
\includegraphics[height=2.5in]{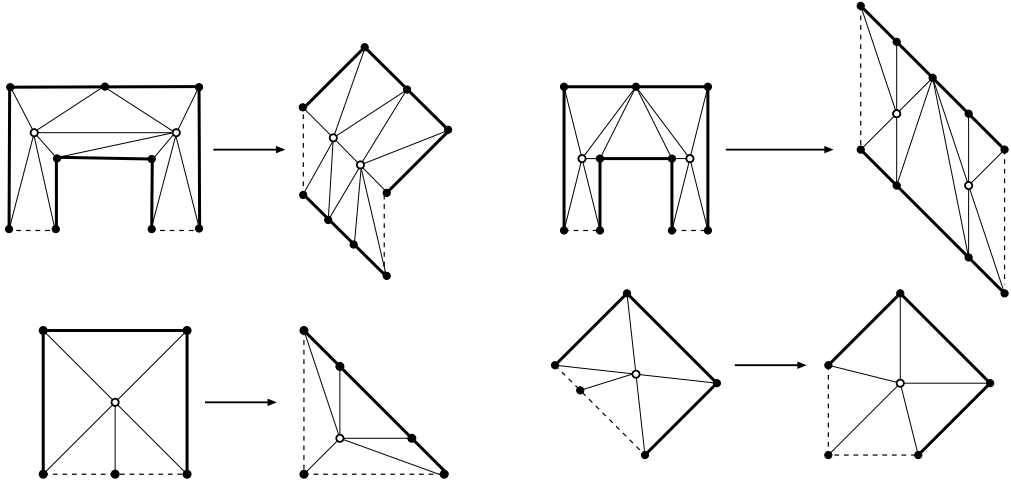}
}
\caption{ \label{TipMap}
The $m$ tip maps also fills in a trapezoid, but is different from 
a trapezoid map because it has to match other tip maps, not 
two rectangle maps of different heights. Again, it is built 
from three types of map: top, bottom and middle. Along its 
left side (for the orientation shown) it matches a $m$-rectangle 
map or the right side of a $(m+1)$-tip map and on the right it 
matches the right side of a $(m-1)$-tip map. The $1$-tip map 
requires a special construction, as shown at lower right.
}
\end{figure}

All the top intervals of the tip trapezoids have the same measure, so 
the tip maps match intervals of the same measure along the vertical 
sides. Tip maps for opposite sides of an interval $K \in {\cal K}$ 
are the same and such  intervals have the same measure from both sides, 
so the maps match here as well.

This completes the construction of the map $\psi: \Omega \to \Omega_4$ and 
the verification that $\psi \circ f \circ \phi^{-1}$ makes $T' $  a 
QC-balanced tree. The construction also clearly shows this map is 
uniformly quasiconformal and is conformal except on a small neighborhood of $T $.
In particular, the quasiconstant is independent
of $\delta$, and as $\delta \to 0$, the support of the dilatation is 
as small as we wish, so that the ``correction'' map obtained from the 
measurable Riemann mapping theorem is as close to the identity 
as we want.  This completes the proof of Theorem \ref{Main}.

\bibliography{truetrees}

\def\cprime{$'$} \def\cprime{$'$} \def\cprime{$'$} \def\cprime{$'$}
  \def\cprime{$'$} \def\cprime{$'$} \def\cprime{$'$} \def\cprime{$'$}
  \def\cprime{$'$} \def\cprime{$'$} \def\cprime{$'$} \def\cprime{$'$}
  \def\cprime{$'$} \def\cprime{$'$}
\begin{thebibliography}{10}

\bibitem{Ahlfors-QCbook}
L.V. Ahlfors.
\newblock {\em Lectures on quasiconformal mappings}, volume~38 of {\em
  University Lecture Series}.
\newblock American Mathematical Society, Providence, RI, second edition, 2006.
\newblock With supplemental chapters by C. J. Earle, I. Kra, M. Shishikura and
  J. H. Hubbard.

\bibitem{classS}
C.J. Bishop.
\newblock Building entire functions by quasiconformal folding.
\newblock preprint, 2012.

\bibitem{Bishop-Pilgrim}
C.J. Bishop and K.M. Pilgim.
\newblock Dynamic dessins are dense.
\newblock preprint, 2013.

\bibitem{MR2053391}
P.L. Bowers and K.~Stephenson.
\newblock Uniformizing dessins and {B}ely\u\i\ maps via circle packing.
\newblock {\em Mem. Amer. Math. Soc.}, 170(805):xii+97, 2004.

\bibitem{MR1196102}
A.~{\`E}. Er{\"e}menko and M.~Yu. Lyubich.
\newblock Dynamical properties of some classes of entire functions.
\newblock {\em Ann. Inst. Fourier (Grenoble)}, 42(4):989--1020, 1992.

\bibitem{MR857196}
L.R. Goldberg and L.~Keen.
\newblock A finiteness theorem for a dynamical class of entire functions.
\newblock {\em Ergodic Theory Dynam. Systems}, 6(2):183--192, 1986.

\bibitem{MR0166776}
F.~Harary, G.~Prins, and W.~T. Tutte.
\newblock The number of plane trees.
\newblock {\em Nederl. Akad. Wetensch. Proc. Ser. A 67=Indag. Math.},
  26:319--329, 1964.

\bibitem{MR2349672}
W.J. Harvey.
\newblock Teichm\"uller spaces, triangle groups and {G}rothendieck dessins.
\newblock In {\em Handbook of {T}eichm\"uller theory. {V}ol. {I}}, volume~11 of
  {\em IRMA Lect. Math. Theor. Phys.}, pages 249--292. Eur. Math. Soc.,
  Z\"urich, 2007.

\bibitem{MR1315551}
P.W. Jones.
\newblock On removable sets for {S}obolev spaces in the plane.
\newblock In {\em Essays on {F}ourier analysis in honor of {E}lias {M}. {S}tein
  ({P}rinceton, {NJ}, 1991)}, volume~42 of {\em Princeton Math. Ser.}, pages
  250--267. Princeton Univ. Press, Princeton, NJ, 1995.

\bibitem{MR1785402}
P.W. Jones and S.K. Smirnov.
\newblock Removability theorems for {S}obolev functions and quasiconformal
  maps.
\newblock {\em Ark. Mat.}, 38(2):263--279, 2000.

\bibitem{MR1746434}
Yu.~Yu. Kochetkov.
\newblock On the geometry of a class of plane trees.
\newblock {\em Funktsional. Anal. i Prilozhen.}, 33(4):78--81, 1999.

\bibitem{MR2476033}
Yu.~Yu. Kochetkov.
\newblock Geometry of planar trees.
\newblock {\em Fundam. Prikl. Mat.}, 13(6):149--158, 2007.

\bibitem{MR2476034}
Yu.~Yu. Kochetkov.
\newblock Planar trees with nine edges: a catalogue.
\newblock {\em Fundam. Prikl. Mat.}, 13(6):159--195, 2007.

\bibitem{MR2411966}
F.~L{\'a}russon and T.~Sadykov.
\newblock Dessins d'enfants and differential equations.
\newblock {\em Algebra i Analiz}, 19(6):184--199, 2007.

\bibitem{Lindsey-Thurston}
K.A. Lindsey and W.P. Thurston.
\newblock Shapes of polynomial {Julia} sets.
\newblock preprint, 2012.

\bibitem{Marshall-Rohde}
D.E. Marshall and S.~Rohde.
\newblock The zipper algorithm for conformal maps and the computation of
  {Shabat} polynomials and dessins.
\newblock manuscript in preparation.

\bibitem{MR1625545}
F.~Pakovitch.
\newblock Combinatoire des arbres planaires et arithm\'etique des courbes
  hyperelliptiques.
\newblock {\em Ann. Inst. Fourier (Grenoble)}, 48(2):323--351, 1998.

\bibitem{MR1834499}
K.M. Pilgrim.
\newblock Dessins d'enfants and {H}ubbard trees.
\newblock {\em Ann. Sci. \'Ecole Norm. Sup. (4)}, 33(5):671--693, 2000.

\bibitem{MR1305390}
L.~Schneps, editor.
\newblock {\em The {G}rothendieck theory of dessins d'enfants}, volume 200 of
  {\em London Mathematical Society Lecture Note Series}.
\newblock Cambridge University Press, Cambridge, 1994.
\newblock Papers from the Conference on Dessins d'Enfant held in Luminy, April
  19--24, 1993.

\bibitem{MR1310587}
G.~Shabat and A.~Zvonkin.
\newblock Plane trees and algebraic numbers.
\newblock In {\em Jerusalem combinatorics '93}, volume 178 of {\em Contemp.
  Math.}, pages 233--275. Amer. Math. Soc., Providence, RI, 1994.

\bibitem{Stahl}
H.R. Stahl.
\newblock Sets of minimal capacity and extremal domains.
\newblock preprint, 2012.

\bibitem{MR819553}
D.~Sullivan.
\newblock Quasiconformal homeomorphisms and dynamics. {I}. {S}olution of the
  {F}atou-{J}ulia problem on wandering domains.
\newblock {\em Ann. of Math. (2)}, 122(3):401--418, 1985.

\bibitem{MR0376411}
D.W. Walkup.
\newblock The number of plane trees.
\newblock {\em Mathematika}, 19:200--204, 1972.

\bibitem{MR2310190}
J.~Wolfart.
\newblock {$ABC$} for polynomials, dessins d'enfants and uniformization---a
  survey.
\newblock In {\em Elementare und analytische {Z}ahlentheorie}, Schr. Wiss. Ges.
  Johann Wolfgang Goethe Univ. Frankfurt am Main, 20, pages 313--345. Franz
  Steiner Verlag Stuttgart, Stuttgart, 2006.

\end{thebibliography}
\bibliographystyle{plain}

\end{document}